\documentclass{amsart}
\usepackage{cases}
\usepackage{graphicx}
\usepackage{amssymb}
\usepackage{amsfonts}
\usepackage{amsmath}
\usepackage{amscd}
 \newtheorem{thm}{Theorem}[section]
 
 \newtheorem{lem}[thm]{Lemma}
 \newtheorem{prop}[thm]{Proposition}
 \theoremstyle{definition}

 \newtheorem{conj}[thm]{Conjecture}
 \theoremstyle{remark}
 \newtheorem{rem}{Remark}
 
 \numberwithin{equation}{section}

\begin{document}
\title[Equivariant Chern numbers and the number of fixed points]
 {Equivariant Chern numbers and the number of fixed points for unitary torus manifolds}

\author{Zhi L\"u and Qiangbo Tan}

\address{School of Mathematical Sciences and The Key Laboratory of Mathematics for Nonlinear Sciences of Ministry of Education,  Fudan
 University, Shanghai, 200433, P.R. China.}
\email{zlu@fudan.edu.cn}
\address{School of Mathematical Sciences,  Fudan
 University, Shanghai, 200433, P.R. China.}
\email{081018011@fudan.edu.cn}
\thanks{Supported in part by grants from  NSFC (No. 10931005) and Shanghai NSF (No. 10ZR1403600), RFDP (No. 20100071110001)}
\email{}


\begin{abstract}
Let~$M^{2n}$ be a unitary torus $(2n)$-manifold, i.e., a $(2n)$-dimensional  oriented stable complex connected closed $T^n$-manifold having a nonempty fixed set. In this paper we show that~$M$ bounds
equivariantly if and only if the equivariant Chern numbers ~$\langle
(c_1^{T^n})^i(c_2^{T^n})^j, [M]\rangle=0$ for all~$i, j\in {\Bbb
N}$, where~$c_l^{T^n}$ denotes the~$l$th equivariant Chern class
of~$M$. As a consequence, we also show that if~$M$ does not bound equivariantly then
the number of fixed points is at least~$\lceil{n\over2}\rceil+1$.
\end{abstract}

\maketitle

\section{Introduction}
Let~$T^n$ denote the torus of rank~$n$. An  {\em  oriented stable
complex closed $T^n$-manifold} is an  oriented closed smooth manifold~$M$ with an effective $T^n$-action such
that its tangent bundle admits a~$T^n$-equivariant stable complex
structure. It is well-known from~\cite{ggk} that the equivariant
cobordism class of an  oriented stable complex closed $T^n$-manifold with
isolated fixed points is completely determined by its equivariant
Chern numbers. In this paper, we shall pay more attention on the  oriented stable complex  (connected) closed $T^n$-manifolds of dimension $2n$ with nonempty fixed set, which are also called the {\em unitary torus manifolds} or {\em unitary toric manifolds} (see \cite{hm} and \cite{m}).
These geometrical objects are the topological analogues of  compact non-singular toric varieties, and constitute a much wider class than that of quasi-toric manifolds introduced by Davis and Januszkiewicz in \cite{dj}. Also, the nonempty fixed set of a unitary torus manifold must be isolated  since the action is assumed to be effective. In this case, we shall show that the equivariant
cobordism class of a unitary torus manifold is determined by only those equivariant Chern
numbers produced by the first and the second equivariant Chern classes. Our result is stated as follows.

\begin{thm}\label{bounds}
Let $M$ be a unitary torus manifold. Then~$M$
bounds equivariantly if and only if the equivariant Chern
numbers $\langle
(c_1^{T^n})^i(c_2^{T^n})^j, [M]\rangle=0$ for all $i, j\in {\Bbb
N}$, where $[M]$ is the fundamental class of $M$ with respect to the given orientation.
\end{thm}

In~\cite{ck}, Kosniowski studied unitary~$S^1$-manifolds and got
some interesting results on the fixed points of the action, where
``unitary'' means that the tangent bundle of~$M$ admits
an~$S^1$-equivariant stable complex structure. In particular, when
the fixed points are isolated, he proposed the following conjecture.

\begin{conj}[Kosniowski]\label{conj} Suppose that~$M^n$ is a unitary $S^1$-manifold with isolated fixed points. If~$M$ does not bound equivariantly then
the number of fixed points is greater than $f(n)$, where $f(n)$ is some linear function.
\end{conj}

\begin{rem}
As was noted by Kosniowski in \cite{ck},  the most likely function is $f(n)={n\over 4}$, so the number of fixed points
of $M^n$ is at least~$[{n\over 4}]+1$.
\end{rem}

With respect to this conjecture, recently some related works have been done.
For example,  Pelayo and Tolman in~\cite{pt} studied compact
symplectic manifolds with symplectic circle actions, and
proved that if the weights induced from the isotropy representations
on the fixed pionts of such a $S^1$-manifold satisfy some subtle condition, then the action
has at least~$n+1$ isolated fixed points. In~\cite{ll}, Ping Li and Kefeng
Liu showed that if~$M^{2mn}$ is an almost complex manifold and there
exists a partition~$\lambda=(\lambda_1,\dots,\lambda_r)$ of weight
$m$ such that the corresponding Chern number~$\langle(c_{\lambda_1}\dots
c_{\lambda_r})^n, [M]\rangle$ is nonzero, then for any~$S^1$-action on~$M$,
it must have at least~$n+1$ isolated fixed points.

\vskip .2cm

In the case of the unitary torus manifolds,  comparing with
Kosniowski's Conjecture~\ref{conj}, we can apply Theorem~\ref{bounds} to obtain the following result:

\begin{thm}\label{number}
Suppose that~$M^{2n}$ is a $(2n)$-dimenional unitary torus manifold. If~$M$ does not bound equivariantly, then
the number of fixed points is at least~$\lceil{n\over2}\rceil+1$, where $\lceil{n\over2}\rceil$ denotes the minimal integer no less than ${n\over 2}$.
\end{thm}

\begin{rem}
It should be interesting to discuss whether there exists  an example of $(2n)$-dimenional unitary torus manifolds,  which
doesn't bound euqivariantly but has exactly~$\lceil{n\over2}\rceil+1$ isolated fixed
points for every $n$.
\end{rem}


\section{Preliminaries}\label{set2}
\subsection{Equivariant Chern characteristic numbers}
The~{\em equivariant Chern characteristic numbers}~$c_\omega^{T^n}(M)$ of
an oriented stable complex closed $T^n$-manifold~$M$ are defined as
$$c_\omega^{T^n}(M)=<(c_1^{T^n})^{i_1}\dots(c_k^{T^n})^{i_k}, [M]>\in H^*(BT^n;{\Bbb Z})$$
where~$\omega=(i_1,\dots,i_k)$ is a multi-index and~$c_l^{T^n}$ is the
$l$th equivariant Chern class of~$M$. Unlike the ordinary Chern characteristic numbers,  these equivariant Chern
characteristic numbers can be nonzero polynomials in $H^*(BT^n;{\Bbb Z})$ if the degree of the product
$\deg(c_1^{T^n})^{i_1}\dots(c_k^{T^n})^{i_k}$
is greater than~$\dim M/2$.

\vskip .2cm

If the oriented stable complex closed $T^n$-manifold $M$ has only isolated fixed
points, then it is known from \cite{ggk} that  at each fixed point $p\in M^{T^n}$ the tangent space
$T_pM$ is equipped with the induced $T^n$-action, orientation and
complex structure, and the $T^n$-equivariant cobordism class of $M$ is determined by
the complex $T^n$-representations $T_pM$ at all $p\in M^{T^n}$ and their orientations. Then Guillemin, Ginzburg and Karshon in \cite{ggk} applied  Atiyah--Bott--Berline--Vergne localization theorem to give the following theorem.

\begin{thm}[Guillemin--Ginzburg--Karshon]\label{iff}
Let~$M$ be an oriented stable complex closed $T^n$-manifold with isolated
fixed points. Then~$M$ bounds equivariantly if and only if all equivariant Chern
characteristic numbers of~$M$ are equal to zero.

\end{thm}

\subsection{Unitary torus manifolds and Atiyah--Bott--Berline--Vergne localization theorem}
Let $M^{2n}$ be a $(2n)$-dimensional unitary torus manifold. Following \cite{m}, we say that a closed, connected, real
codimension-2 submanifold of $M^{2n}$ is called {\em characteristic} if it is a fixed set component by a certain circle subgroup of $T^n$ and contains at least one $T^n$-fixed point. Then $M^{2n}$ has finitely many such characteristic submanifolds. By $M_i, i\in [m]=\{1, ..., m\}$ we denote all characteristic submanifolds of $M^{2n}$, and by $\zeta_i$ denote the corresponding normal bundle over $M_i$, and by $T_i$ denote the circle subgroup fixing $M_i$ pointwise. Then, for each $p\in M^{T^n}$, we can write the tangent $T^n$-representation at $p$ as
$$T_pM=\bigoplus_{i\in I(p)}\zeta_i|_p$$
where $I(p)=\{i|p\in M_i\}\subset [m]$ and $\zeta_i|_p$ is the restriction of $\zeta_i$ to $p$. So $|I(p)|=n$. Each
$M_i$ may define an element $\lambda_i$ in the equivariant cohomology $H^2_{T^n}(M;{\Bbb Z})$. Actually, the inclusion $M_i\hookrightarrow M$ may induce an equivariant Gysin homomorphism: $H^*_{T^n}(M_i;{\Bbb Z})\longrightarrow
H^{*+2}_{T^n}(M;{\Bbb Z})$, so that $\lambda_i\in H^2_{T^n}(M;{\Bbb Z})$ can be chosen as the image of the identity in $H^0_{T^n}(M_i;{\Bbb Z})$. It was shown in \cite[Theorem 3.1]{m} that the total equivariant Chern characteristic
class $c^{T^n}(TM)$ of the tangent bundle $TM$ of $M$ can be expressed as
\begin{equation*}
c^{T^n}(TM)=\prod_{i\in [m]}(1+\lambda_i)
\end{equation*}
in $\widehat{H}^*_{T^n}(M;{\Bbb Z})=H^*_{T^n}(M;{\Bbb Z})/S\text{-torsion}$ where $S$ is the subset of $H^*(BT^n;{\Bbb Z})$ generated multiplicatively by nonzero elements of $H^2(BT^n;{\Bbb Z})$.
It is well-known that the restriction $\lambda_i|_p$ can be regarded as the top equivariant Chern class of $\zeta_i|_p$.
Hence, the total equivariant Chern characteristic class of the vector bundle $T_pM\longrightarrow \{p\}$ is
\begin{equation*}
c^{T^n}(T_pM)=c^{T^n}(TM)|_p=\prod_{i\in I(p)}(1+\lambda_i|_p).
\end{equation*}
In particular, Masuda in \cite{m} also showed the following result, which will be very useful in our discussion later. \begin{lem}[{\cite[Lemma 1.3(1)]{m}}]\label{basis}
 $\{\lambda_i|_p\big| i\in I(p)\}$ forms a basis of $H^2_{T^n}(\{p\};{\Bbb Z})\cong
H^2(BT^n;{\Bbb Z})$.
\end{lem}

On the other hand, the normal bundle to $p$ in $M^{2n}$ is $T_pM$ with the orientation inherited from $M^{2n}$.
Thus, the equivariant Euler class $e^{T^n}(T_pM)$ of this bundle is $\pm c_n^{T^n}(T_pM)=\pm\prod_{i\in I(p)}\lambda_i|_p$, where the sign is positive if the orientation of $T_pM$ agrees with the complex orientation and negative otherwise.

\vskip .2cm

Each $c^{T^n}(T_pM)=\prod_{i\in I(p)}(1+\lambda_i|_p)=1+\sigma_1(p)+\cdots +\sigma_n(p)$
determines a collection $\sigma(p)=(\sigma_1(p), ..., \sigma_n(p))$, where $\sigma_j(p)$ denotes the $j$th elementary symmetric function over $n$ variables $\lambda_i|_p, i\in I(p)$. Clearly, $\sigma(p)$
 determines the representation $T_pM$, but not the orientation
of~$T_pM$ inherited from~$M$.

\vskip .2cm

Now choose a basis $\Lambda=\{\alpha_1, ..., \alpha_n\}$ in $H^2(BT^n;{\Bbb Z})$. Then we obtain a collection
$\sigma(\Lambda)=(\sigma_1(\Lambda), ..., \sigma_n(\Lambda))$, where $\sigma_i(\Lambda)$ means the $i$th elementary symmetric function $\sigma_i(\alpha_1, ..., \alpha_n)$.
Set
\begin{eqnarray*}
m_{\sigma(\Lambda)}&=&\sharp\{p\in
M^{T^n}\big|\sigma(p)=\sigma(\Lambda), e^{T^n}(T_pM)=\sigma_n(\Lambda)\}\\
&&-\sharp\{p\in
M^{T^n}\big|\sigma(p)=\sigma(\Lambda), e^{T^n}(T_pM)=-\sigma_n(\Lambda)\}
\end{eqnarray*}
where we let $m_{\sigma(\Lambda)}=0$ if $\sigma(\Lambda)$ does not occur as the collection of symmetric functions for any of the $T_pM$'s.  With the above understood, now we can state the
Atiyah--Bott--Berline--Vergne localization theorem in our case as
follows:

\begin{thm}[A--B--B--V localization theorem]\label{ABBV}
Let~$M^{2n}$ be a $(2n)$-dimensional unitary torus manifold.  Then
$$c_\omega^{T^n}(M)=\sum_{p\in M^{T^n}}{\sigma_1(p)^{i_1}\dots\sigma_n(p)^{i_n}\over\pm\sigma_n(p)}
=\sum_{\sigma(\Lambda)}m_{\sigma(\Lambda)}\sigma_1(\Lambda)^{i_1}\dots\sigma_n(\Lambda)^{i_n-1}$$
where~$\omega=(i_1,\dots,i_n)$ is a multi-index.
\end{thm}

\section{Proofs of main results}\label{proof}
First we prove two lemmas which will be used in the proof of
Theorem~\ref{bounds}. Let~$M^{2n}$ be a $(2n)$-dimensional unitary torus manifold and let $p,q\in
M^{T^n}$ be two fixed points.
\begin{lem}\label{sigma1}
 If $\sigma_1(p)=\sigma_1(q)$
and~$\sigma_n(p)=\pm\sigma_n(q)$ then $\sigma(p)=\sigma(q)$.
\end{lem}

\begin{proof}
 If $\sigma_n(p)=\pm\sigma_n(q)$, then $\prod_{i\in I(p)}\lambda_i|_p=\pm \prod_{i\in I(q)}\lambda_i|_q$.
 So, by Lemma~\ref{basis} we have that $\{\lambda_i|_p\big| i\in I(p)\}=\{\varepsilon_i\lambda_i|_q\big|i\in I(q)\}$
where~$\varepsilon_i=\pm1$. Furthermore, if $\sigma_1(p)=\sigma_1(q)$, then
$$\sigma_1(p)=\sum_{i\in I(p)}\lambda_i|_p=\sum_{i\in I(q)}\varepsilon_i\lambda_i|_q=\sum_{i\in I(q)}\lambda_i|_q=\sigma_1(q)$$
 so $\sum_{i\in I(q)}(1-\varepsilon_i)\lambda_i|_q=0$.  This implies that
$\varepsilon_i=1$ for all $i\in I(q)$ since $\lambda_i|_q, i\in I(q)$ are linearly independent, and the lemma then follows.
\end{proof}

\begin{lem}\label{sigma2}
$\sigma(p)=\sigma(q)$ if and only if $\sigma_1(p)=\sigma_1(q)$ and
$\sigma_2(p)=\sigma_2(q)$.
\end{lem}

\begin{proof} It suffices to show that $\sigma(p)=\sigma(q)$ if $\sigma_1(p)=\sigma_1(q)$ and
$\sigma_2(p)=\sigma_2(q)$.
 Consider $s_2(p)=\sum_{i\in I(p)}(\lambda_i|_p)^2$
and~$s_2(q)=\sum_{i\in I(q)}(\lambda_i|_q)^2$.
If~$\sigma_1(p)=\sigma_1(q)$ and~$\sigma_2(p)=\sigma_2(q)$,
then~$s_2(p)=s_2(q)$ since $s_2=\sigma_1^2-2\sigma_2$ by \cite{ms}.
Since both $\{\lambda_i|_p\big| i\in I(p)\}$ and $\{\lambda_i|_q\big|i\in I(q)\}$ are two bases of $H^2(BT^n;{\Bbb Z})$
by Lemma~\ref{basis}, there is an~$n\times n$
non-degenerate ${\Bbb Z}$-matrix $A$ such that $$(\lambda_i|_p\big| i\in I(p))=(\lambda_i|_q\big|i\in I(q))A.$$ Moreover,
we have that $$s_2(p)-s_2(q)=(\lambda_i|_q\big|i\in I(q))(AA^\top-E_n)(\lambda_i|_q\big|i\in I(q))^\top=0$$
so we conclude that $AA^\top=E_n$, where $E_n$ is the identity matrix. This implies that each row of~$A$ contains
only one~$\pm1$ and the other elements in this row are all 0.
Hence $$\sigma_n(p)=\pm \prod_{i\in I(q)}\lambda_i|_q=\pm\sigma_n(q)$$ and then the proof is
completed by Lemma~\ref{sigma1}.
\end{proof}

Let $\{\sigma_1(p)\big|p\in M^{T^n}\}=\{\tau_1,\dots,\tau_s\}$ and
$\{\sigma_2(p)\big|p\in M^{T^n}\}=\{\eta_1,\dots,\eta_u\}$. Set $$\mathcal{A}_k=\{p\in M^{T^n}\big|\sigma_1(p)=\tau_k\}$$ for $1\leq
k\leq s$ and $$\mathcal{B}_l=\{p\in M^{T^n}\big|\sigma_2(p)=\eta_l\}$$ for $1\leq l\leq u$.
Then $|M^{T^n}|=\sum_{k=1}^s|\mathcal{A}_k|=\sum_{l=1}^u|\mathcal{B}_l|$.
\vskip .2cm

\begin{proof}[Proof of Theorem~\ref{bounds}]
By Theorem~\ref{iff} it suffices to prove that if the equivariant Chern numbers $\langle
(c_1^{T^n})^i(c_2^{T^n})^j, [M]\rangle=0$ for all $i, j\in {\Bbb
N}$,  then $M^{2n}$ bounds equivariantly.
Now suppose $\langle
(c_1^{T^n})^i(c_2^{T^n})^j, [M]\rangle=0$ for all $i, j\in {\Bbb
N}$. By Theorem~\ref{ABBV}, we can write these equivariant Chern numbers in the following way:
\begin{equation}
\langle (c_1^{T^n})^i(c_2^{T^n})^j,
[M]\rangle=\sum_{k=1}^s\tau_k^i\sum_{l\in \mathcal{L}_k}\eta_l^j\sum_{\sigma_1=
\tau_k\atop \sigma_2=\eta_l}{m_{\sigma}\over\sigma_n}
\end{equation}
where $\mathcal{L}_k=\{l\big|\mathcal{A}_k\cap\mathcal{B}_l\not=\emptyset, 1\leq l\leq u\}$. Obviously, $|\mathcal{L}_k|\leq |\mathcal{A}_k|$ for every $k$.
Let~$i$ vary in the range $0,1, \dots, s-1$. Then~$(\tau_k^i)$ is
an~$s\times s$ van der Monde matrix, so for each $k$,
$$\sum_{l\in \mathcal{L}_k}\eta_l^j\sum_{\sigma_1=\tau_k\atop
\sigma_2=\eta_l}{m_{\sigma}\over\sigma_n}=0.$$
Next, let~$j$ vary in the range $0,1, \dots, |\mathcal{L}_k|-1$. Then~$(\eta_l^j)$ is
a~$|\mathcal{L}_k|\times |\mathcal{L}_k|$ van der Monde matrix, hence
$$\sum_{\sigma_1=\tau_k\atop \sigma_2=\eta_l}{m_{\sigma}\over\sigma_n}=0.$$
Furthermore, by Lemma~\ref{sigma2} we have that ${m_{\sigma}\over\sigma_n}=\sum_{\sigma_1=\tau_k\atop \sigma_2=\eta_l}{m_{\sigma}\over\sigma_n}=0$ so $m_\sigma=0$ for all $\sigma$.
Thus, by Theorem~\ref{ABBV}, all equivariant Chern characteristic numbers of~$M$
are equal to zero, as desired.
\end{proof}

Now we focus on the proof of Theorem~\ref{number}. First we give a
general result.\vskip .2cm
\begin{prop}\label{equ}
If~$s+2\max\limits_{ 1\leq k\leq s}\{|\mathcal{A}_k|\}-3<n$ or~$2u+\max\limits_{1\leq l\leq u}\{|\mathcal{B}_l|\}-3<n$,
then~$M$ bounds equivariantly.
\end{prop}
\begin{proof} In a similar way to the proof of Theorem~\ref{bounds}, we can write the equivariant Chern numbers $\langle (c_1^{T^n})^i(c_2^{T^n})^j,
[M]\rangle$ in the following two ways:
\begin{equation}\label{tau}
\langle (c_1^{T^n})^i(c_2^{T^n})^j,
[M]\rangle=\sum_{k=1}^s\tau_k^i\sum_{l\in \mathcal{L}_k}\eta_l^j\sum_{\sigma_1=\tau_k\atop\sigma_2=\eta_l}{m_{\sigma}\over\sigma_n}
\end{equation}
and
\begin{equation}\label{eta}
\langle (c_1^{T^n})^i(c_2^{T^n})^j,
[M]\rangle=\sum_{l=1}^u\eta_l^j\sum_{k\in \mathcal{K}_l}\tau_k^i\sum_{\sigma_1=\tau_k\atop\sigma_2=\eta_l}{m_{\sigma}\over\sigma_n}
\end{equation}
where $\mathcal{L}_k=\{l\big|\mathcal{A}_k\cap\mathcal{B}_l\not=\emptyset, 1\leq l\leq u\}$ as before and $\mathcal{K}_l=\{k\big|
\mathcal{A}_k\cap\mathcal{B}_l\not=\emptyset, 1\leq k\leq s\}$, satisfying that $|\mathcal{L}_k|\leq |\mathcal{A}_k|$ for every $k$ and $|\mathcal{K}_l|\leq |\mathcal{B}_l|$ for every $l$.
\noindent We note that  if $i+2j<n$, then $\langle (c_1^{T^n})^i(c_2^{T^n})^j,
[M]\rangle=0$. If~$s+2\max\limits_{1\leq k\leq s}\{|\mathcal{A}_k|\}-3<n$, then we can let $i$
vary in the range $0,1, \dots,s-1$ and for every $k$, let $j$ vary in the
range $0,1, \dots, |\mathcal{L}_k|-1\leq \max\limits_{1\leq k\leq s}\{|\mathcal{A}_k|\}-1$ in~Equation~(\ref{tau}).
Similarly,
if $2u+\max\limits_{1\leq l\leq u}\{|\mathcal{B}_l|\}-3<n$, then we can let $j$ vary in the
range $0,1,\dots,u-1$ and for every $l$, let $i$ vary in the
range $0,1, \dots, |\mathcal{K}_l|-1\leq \max\limits_{1\leq l\leq u}\{|\mathcal{B}_l|\}-1$ in Equation~(\ref{eta}). Using the proof method of Theorem~\ref{bounds} as above, we can get van der
Monde matrices, which imply that~$m_\sigma=0$ for all~$\sigma$, and
hence $\langle (c_1^{T^n})^i(c_2^{T^n})^j, [M]\rangle=0$ for all $i, j\in {\Bbb N}$. Therefore, $M$
bounds equivariantly by Theorem~\ref{bounds}.
\end{proof}

\begin{lem}\label{opt}
Let $a_1,\dots,a_r$ be positive integers. If $a_1+\dots+a_r=\ell$,
then $r+2\max\{a_i|1\leq i\leq r\}\leq 2\ell+1$.
\end{lem}
\begin{proof}
Obviously, $\text{max}\{a_i|1\leq i\leq r\}\leq \ell-r+1$, and the equation
holds if and only if there is only some one $a_i=\ell-r+1$ and all others are equal to 1. Then
we have the required inequality  $r+2\max\{a_i|1\leq i\leq r\}\leq 2\ell+1$, where the
last equation holds if and only if $r=1$.
\end{proof}

\begin{proof}[Proof of Theorem~\ref{number}]
If~$|M^{T^n}|=|\mathcal{A}_1|+\dots+|\mathcal{A}_s|<{n\over 2}+1$, then by
Lemma~\ref{opt}, $s+2\max\limits_{1\leq k\leq s}\{|\mathcal{A}_k|\}\leq 2|M^{T^n}|+1<n+3$, so~$M$
bounds equivariantly by Proposition~\ref{equ}.
\end{proof}

\begin{rem} Let us look at the case in which $M$ doesn't bound equivariantly and $|M^{T^n}|=\lceil{n\over
2}\rceil+1$.
When $n$ is even,  we have   $s+2\max\limits_{1\leq k\leq s}\{|\mathcal{A}_k|\}=2|M^{T^n}|+1$ by Proposition~\ref{equ} and Lemma~\ref{opt}. This implies that $s=1$
by the proof of Lemma~\ref{opt}, which means that all~$\sigma_1$ are
the same. When $n$ is odd,  we see  that $n+3\leq s+2\max\limits_{1\leq k\leq s}\{|\mathcal{A}_k|\}
\leq 2|M^{T^n}|+1=n+4$. An easy argument shows that $n+3= s+2\max\limits_{1\leq k\leq s}\{|\mathcal{A}_k|\}$ is impossible, so we must have  $s+2\max\limits_{1\leq k\leq s}\{|\mathcal{A}_k|\}=
 2|M^{T^n}|+1$. Thus, in this case $s$ must be 1 and then all~$\sigma_1$ are
the same, too. Moreover, in a similar way to the proof of Theorem~\ref{bounds}, we can show easily that $|M^{T^n}|=u$ so all $\sigma_2$ are distinct. These observations  seemingly imply the existence of a nonbounding unitary torus manifold $M^{2n}$ with $|M^{T^n}|=\lceil{n\over
2}\rceil+1$. Indeed, we can see an example in the case $n=1$, as shown in \cite[Theorem 5]{ck}.
\end{rem}
Finally we conclude this paper with the following conjecture:

\begin{conj}
$\lceil{n\over
2}\rceil+1$ is the best possible lower bound of the number of fixed points for $(2n)$-dimensional nonbounding unitary torus manifolds.
\end{conj}


\begin{thebibliography}{99}
\bibitem[DJ]{dj} M. W. Davis and T. Januszkiewicz, {\em
Convex polytopes, Coxeter orbifolds and torus Actions}, Duke
Mahtematicial Journal,  {\bf 62} (1991), 417--451.
\bibitem[GGK]{ggk} V. Guillemin, V. Ginzburg and Y. Karshon, {\em Moment Maps, Cobordisms, and Hamiltonian Group Actions}, Mathematical Surveys and Monographs, {\bf 98},  American Mathematical Society, Providence, RI, 2002.
  \bibitem[HM]{hm}  A. Hattori and  M.  Masuda, {\em  Theory of multi-fans}, Osaka J. Math. {\bf 40} (2003), 1--68.
\bibitem[K]{ck} Czes Kosniowski, {\em Some formulae and conjectures associated with circle
actions}, Topology Symposium, Siegen 1979 (Proc. Sympos., Univ.
Siegen, Siegen, 1979), pp331--339, Lecture Notes in Math., {\bf 788},
Springer, Berlin, 1980.
\bibitem[LL]{ll} Ping Li and Kefeng Liu, {\em Some remarks on circle action on
manifolds}, arXiv: 1008.4820, to appear in Mathematical Research Letters.
\bibitem[M]{m} M. Masuda, {\em Unitary toric manifolds, multi-fans and equivariant index}, Tohoku Math. J.
{\bf 51} (1999), 237--265.
\bibitem[MS]{ms} J.W. Milnor and J.D. Stasheff, {\em Characteristic Classes}, Ann. of Math. Stud., vol. {\bf 76}, Princeton University Press/University of Tokyo Press, Princeton, NJ/Tokyo, 1974.
 \bibitem[PT]{pt} A. Pelayo and S. Tolman, {\em Fixed points of symplectic
periodic flows}, arXiv: 1003.4787, to appear in Ergodic Theory and
Dynamical Systems.
 \end{thebibliography}
\end{document}